\documentclass[12pt]{amsart}

\usepackage{amssymb}

\usepackage[cmtip,all]{xy}

\usepackage{verbatim}
\usepackage{upref}
\usepackage[usenames,dvipsnames]{color}
\usepackage{url}
\usepackage[normalem]{ulem}

\newtheorem{thm}{Theorem}[section]
\newtheorem*{thm*}{Theorem}
\newtheorem{lem}[thm]{Lemma}
\newtheorem{cor}[thm]{Corollary}
\newtheorem{prop}[thm]{Proposition}
\newtheorem*{prop*}{Proposition}

\theoremstyle{definition}
\newtheorem{defn}[thm]{Definition}
\newtheorem*{defn*}{Definition}

\newtheorem{ex}[thm]{Example}
\newtheorem{notn}[thm]{Notation}

\newtheorem*{notn*}{Notation}

\newtheorem*{hyp*}{Hypothesis}

\theoremstyle{remark}
\newtheorem{rem}[thm]{Remark}
\newtheorem*{rem*}{Remark}

\numberwithin{equation}{section}

\newcommand{\secref}[1]{Section~\textup{\ref{#1}}}

\newcommand{\thmref}[1]{Theorem~\textup{\ref{#1}}}
\newcommand{\corref}[1]{Corollary~\textup{\ref{#1}}}
\newcommand{\lemref}[1]{Lemma~\textup{\ref{#1}}}
\newcommand{\propref}[1]{Proposition~\textup{\ref{#1}}}

\newcommand{\defnref}[1]{Definition~\textup{\ref{#1}}}
\newcommand{\remref}[1]{Remark~\textup{\ref{#1}}}

\newcommand{\exref}[1]{Example~\textup{\ref{#1}}}

\newcommand{\righttext}[1]{\quad\text{#1 }}

\renewcommand{\)}{\textup)}
\newcommand{\ie}{\emph{i.e.}}

\newcommand{\N}{\mathbb N}
\newcommand{\Z}{\mathbb Z}

\newcommand{\R}{\mathbb R}
\newcommand{\C}{\mathbb C}

\newcommand{\F}{\mathbb F}

\newcommand{\KK}{\mathcal K}

\newcommand{\JJ}{\mathcal J}

\newcommand{\LL}{\mathcal L}

\newcommand{\EE}{\mathcal E}

\newcommand{\RR}{\mathcal R}

\newcommand{\MM}{\mathcal M}
\newcommand{\NN}{\mathcal N}

\renewcommand{\AA}{\mathcal A}

\newcommand{\Chi}{\raisebox{2pt}{\ensuremath{\chi}}}
\renewcommand{\epsilon}{\varepsilon}

\DeclareMathOperator{\ad}{Ad}
\DeclareMathOperator*{\spn}{span}
\DeclareMathOperator*{\clspn}{\overline{\spn}}

\newcommand{\id}{\text{\textup{id}}}

\newcommand{\<}{\langle}
\renewcommand{\>}{\rangle}
\newcommand{\then}{\ensuremath{\Rightarrow}}
\newcommand{\minus}{\setminus}
\newcommand{\under}{\backslash}
\newcommand{\inv}{^{-1}}

\renewcommand{\bar}{\overline}
\newcommand{\what}{\widehat}
\newcommand{\wilde}{\widetilde}
\newcommand{\rt}{\textup{rt}}
\newcommand{\ann}{^\perp}
\newcommand{\pann}{{}\ann}

\newcommand{\si}{_\text{si}}

\begin{document}

\title{Properness conditions for actions and coactions}



\author[Kaliszewski]{S. Kaliszewski}
\address{School of Mathematical and Statistical Sciences
\\Arizona State University
\\Tempe, Arizona 85287}
\email{kaliszewski@asu.edu}

\author[Landstad]{Magnus~B. Landstad}
\address{Department of Mathematical Sciences\\
Norwegian University of Science and Technology\\
NO-7491 Trondheim, Norway}
\email{magnusla@math.ntnu.no}

\author[Quigg]{John Quigg}
\address{School of Mathematical and Statistical Sciences
\\Arizona State University
\\Tempe, Arizona 85287}
\email{quigg@asu.edu}

\subjclass[2000]{Primary  46L55}

\keywords{crossed product,
action,
proper action,
coaction, 
Fourier-Stieltjes algebra,
exact sequence,
Morita compatible}

\dedicatory{Dedicated to R.~V. Kadison --- teacher and inspirator}

\date{March 31, 2015} 

\begin{abstract}
Three properness conditions for actions of locally compact groups on $C^*$-algebras are studied,
as well as their dual analogues for coactions.
To motivate the properness conditions for actions, the commutative cases (actions on spaces) are surveyed; here the conditions are known:
proper, locally proper, and pointwise properness, although the latter property has not been so well studied in the literature. The basic theory of these properness conditions is summarized, with somewhat more attention paid to pointwise properness.
$C^*$-characterizations of the properties are proved,
and applications to $C^*$-dynamical systems are examined.
This paper is partially expository, but some of the results are believed to be new.
\end{abstract}
\maketitle


\section{Introduction}\label{intro}

In our recent study of
$C^*$-covariant systems $(A,G,\alpha)$ 
and
crossed product algebras between the full crossed product $A\rtimes_\alpha G$ and  the regular crossed product $A\rtimes_{\alpha,r} G$, it turns out that various generalizations of the concept of proper actions of $G$ 
play an important role. We 
therefore 
start by taking
a closer look at this concept, and it turns out that even for  a classical  action of $G$  on a space $X$ we made what we believe to be new discoveries.

Classically (going back to Bourbaki \cite{topologie}), a $G$-space $X$ is called \emph{proper} if the map from $G\times X$ to $X\times X$ given by
\[
(s,x)\mapsto (x,sx)
\]
is proper, i.e., inverse images of compact sets are compact.

We call the action \emph{pointwise proper} if the map from $G$ to $X$ given by
\[
s\mapsto sx
\]
is proper for each $x\in X$. 

There is also an intermediate property: $X$ is  \emph{locally proper} if each point of $X$ has a $G$-invariant neighbourhood 
on which $G$ acts properly.

Apparently the above terminology is not completely standard.
For a discrete group, \cite{varadarajan} uses the terms
\emph{discontinuous}, \emph{properly discontinuous}, and \emph{strongly properly discontinuous}
instead of
pointwise proper, locally proper, and proper, respectively.
Palais uses \emph{Cartan} instead of locally proper.
And \cite{koszul} uses the terms $P_2$, $P_1$, and $P$, respectively.
A characteristic property of properness (see \lemref{compact wandering} below) is sometimes referred to as ``compact sets are wandering''.

It is folklore that for proper $G$-spaces $X$ the full crossed product $C_0(X) \rtimes_{\alpha} G$ 
 is isomorphic to  the reduced crossed product $C_0(X) \rtimes_{\alpha,r} G$ (see \cite{phillipsproper} for the second countable case).
In \propref{locally regular} (perhaps also folklore) we show that this carries over to locally proper actions.
We will show in \thmref{pointwise regular} (believed to be new) that this is true also if $X$ is first countable, but the action is only assumed to be pointwise proper.

We propose the following as natural generalizations of properness to a general $C^*$-covariant system $(A,G,\alpha)$:

\begin{defn*}
\quad

\begin{itemize}
\item $(A,G,\alpha)$ is  {\it s-proper} if for all $a,b\in A$ the map
\[
g\mapsto \alpha_g(a)b \text{ is in } C_0(G,A).
\]
\item
$(A,G,\alpha)$ is  {\it w-proper} if for all $a\in A$, $\phi\in A^*$ the map
\[
g\mapsto \phi(\alpha_g(a)) \text{ is in } C_0(G).
\]
\end{itemize} 
\end{defn*}

This is consistent with the classical case, for $A=C_0(X)$ we have 
\begin{align*}
(X,G) \text{ is proper} &\iff (C_0(X),G)  \text{ is s-proper } \\
(X,G) \text{ is pointwise proper} &\iff (C_0(X),G)  \text{ is w-proper. } 
\end{align*}

One indication that w-properness is an interesting property is the following

\begin{prop*} Suppose
$(A,G,\alpha)$ is   w-proper, $\pi$ a representation of $A$, and $s\mapsto U_s$  a continuous map into the unitaries
\(but not necessarily a homomorphism\) such that $\pi(\alpha_s(a))=U_s\pi(a)U_s^*$. Then for all $\xi,\eta$ in the Hilbert space the
coefficient function
$s\mapsto \<U_s\xi,\eta\>$ is in $C_0(G)$.
\end{prop*}
We treat the classical situation of a $G$-space $X$ in Sections~\ref{spaces} and \ref{c star ramifications}, and discuss general $C^*$-covariant systems in \secref{c star general}.

For a $C^*$-covariant system $(A,G,\alpha)$,
there are various definitions of properness (by Rieffel and others) involving some integrability properties. We show in \secref{rieffel proper} that they  imply s- or w-properness. The main purpose of these integrability properties is to define a suitable fixed point algebra in $M(A)$, so our properness definitions are too general for this purpose.

The natural dual concept of a $C^*$-covariant system
is that of  a \emph{coaction}. As we briefly describe in \secref{coactions}, it turns out that 
s- and w-properness can be defined in a similar way for coactions, and we describe some of the relevant results.

In \secref{E} we describe a general construction of crossed product algebras between
$A\rtimes_{\alpha} G$  and  $A\rtimes_{\alpha,r} G$.  We claim that the interesting ones are obtained by first taking as our group  $C^*$-algebra  $C^*(G)/I$ where $I$ is a \emph{small} ideal of $C^*(G)$ (\ie\  $I$ is $\delta_G$-invariant and contained in the kernel of the regular representation $\lambda$ of $C^*(G)$).
We showed in [KLQ13] that  $I$ is a  small ideal of $C^*(G)$ if and only if
 the annihilator $E=I\ann$ in $B(G)$ is a \emph{large} ideal,
in the sense that it is a nonzero, weak* closed, and $G$-invariant ideal of the Fourier-Stieltjes algebra
$B(G)$. There are various interesting examples, see  [BG] and  [KLQ13].

Now to a $C^*$-covariant system $(B,G,\alpha)$ and $E$ as above one can define an $E$-crossed product $B\rtimes_{\alpha,E} G$ between the full  and the reduced crossed product. In  [KLQ13] we show that if the coaction is w-proper then there is a Galois theory describing these crossed products.

Finally we mention the work by  Kirchberg, Baum, Guentner, and Willet \cite{bgwexact} on the Baum-Connes conjecture. They have shown that
there is a unique minimal exact and Morita compatible functor
that assigns to a $C^*$-covariant system $(A,G,\alpha)$ a $C^*$-algebra between $A\rtimes_\alpha G$ and $A\rtimes_{\alpha,r} G$.
At least one of the authors doubts that this minimal functor is an $E$-crossed product for some large ideal $E$, although this remains an open problem.

In Sections~\ref{spaces}--\ref{full=reduced} we give a fairly detailed exposition, in particular proofs of results we believe to be new. Sections~\ref{coactions}--\ref{E} will  be more
descriptive, referring to the literature for details and proofs.

\section{Actions on spaces}\label{spaces}

Throughout, $G$ will be a locally compact group, 
$A$ will be a $C^*$-algebra,
and $X$ will be a locally compact Hausdorff space.
We will be concerned with actions $\alpha$ 
of 
$G$
on 
$A$,
and we just say $(A,\alpha)$ is an action
since the group $G$ will typically be fixed.
If $G$ acts on $X$ then
we sometimes call $X$ a \emph{$G$-space},
and
the \emph{associated action} $(C_0(X),\alpha)$ is defined by
\[
\alpha_s(f)(x)=f(s\inv x)\righttext{for}s\in G,f\in C_0(X),x\in X.
\]
Recall that, since the map $(s,x)\mapsto sx$ from $G\times X$ to $X$ is continuous,
the associated action $\alpha$ is \emph{strongly continuous} in the sense that
for all $f\in C_0(X)$ the map $s\mapsto \alpha_s(f)$ from $G$ to $C_0(X)$ is continuous for the uniform norm.

The following notation is borrowed from Palais \cite{palais}:

\begin{notn}
If $G$ acts on $X$,
then for two subsets $U,V\subset X$
we define
\[
((U,V))=\{s\in G:sU\cap V\ne\varnothing\}.
\]
\end{notn}
Note that if $U$ and $V$ are compact then $((U,V))$ is closed in $G$.

Much of the following discussion of actions on spaces is well-known;
we present it in a formal way for convenience.
We make no attempt at completeness,
but at the same time we include many proofs to make this exposition self-contained.
When a result can be explicitly found in \cite{palais}, we give a precise reference,
but lack of such a reference should not be taken as any claim of originality.
In much of the literature on proper actions the spaces are only required to be Hausdorff, or completely regular;
in the proofs we will take full advantage of our assumption that our spaces are locally compact Hausdorff.

\begin{defn}\label{proper defn}
A $G$-space $X$
is \emph{proper} if the map
$\phi:X\times G\to X\times X$ defined by
$\phi(x,s)=(x,sx)$
is proper,
i.e., inverse images of compact sets are compact.
\end{defn}

The following is routine,
and explains why properness is sometimes referred to as ``compact sets are wandering''
(e.g., \cite[Situation~2]{rieffelapplications}):
\begin{lem}\label{compact wandering}
A $G$-space $X$
is proper if and only if for every compact $K\subset X$ the set $((K,K))$ is compact,
equivalently for every compact $K,L\subset X$ the set $((K,L))$ is compact.
\end{lem}

\begin{ex}
If $H$ is a closed subgroup of $G$, then it is an easy exercise that
the action of $G$ on the homogeneous space $G/H$ by translation is proper if and only if $H$ is compact.
\end{ex}

The following result is contained in \cite[Theorem~1.2.9]{palais}.

\begin{prop}
A $G$-space $X$ is proper if and only if for all $x,y\in X$ there are neighborhoods $U$ of $x$ and $V$ of $y$ such that $((U,V))$ is relatively compact.
\end{prop}

\begin{proof}
One direction is obvious, since if the action is proper we only need to choose the neighborhoods $U$ and $V$ to be compact.

Conversely, assume the condition involving pairs of points $x,y$,
and let $K\subset X$ be compact.
To show that $((K,K))$ is compact, we will prove that any net $\{s_i\}$ in $((K,K))$ has a convergent subnet.
For every $i$ we can choose $x_i\in K$ such that $s_ix_i\in K$.
Passing to subnets and relabeling, we can assume that $x_i\to x$ and $s_ix_i\to y$ for some $x,y\in K$.
By assumption we can choose compact neighborhoods $U$ of $x$ and $V$ of $y$ such that $((U,V))$ is compact.
Without loss of generality, for all $i$ we have $x_i\in U$ and $s_ix_i\in V$, and hence
$s_i\in ((U,V))$.
Thus $\{s_i\}$ has a convergent subnet by compactness.
\end{proof}

\begin{defn}
A $G$-space $X$ is \emph{locally proper} if 
it is a union of open $G$-invariant sets on which $G$ acts properly.
\end{defn}

Palais uses the term \emph{Cartan} instead of locally proper.
The forward direction of the following result is \cite[Proposition~1.2.4]{palais}.

\begin{lem}\label{UU}
A $G$-space $X$ is locally proper if and only if
every
$x\in X$ 
has a
neighborhood $U$ 
such that $((U,U))$ is compact.
\end{lem}

\begin{proof}
First assume 
that the action is locally proper,
and let $x\in X$.
Choose an open $G$-invariant set $V$ containing $x$ on which $G$ acts properly.
Then choose a compact neighborhood $U$ of $x$ contained in $V$.
Then $((U,U))$ is compact by properness.

Conversely, assume the condition involving compact sets $((U,U))$.
Choose an open neighborhood $V$ of $x$ such that $((V,V))$ is relatively compact,
and let $U=GV$.
We will show that the action of $G$ on $U$ is proper.
Let $y,z\in U$.
Choose $s,t\in G$ such that $y\in sV$ and $z\in tV$.
Then we have neighborhoods $sV$ of $y$ and $tV$ of $z$, and
\[
((sV,tV))=t((V,V))s\inv
\]
is relatively compact.
\end{proof}

The following result displays a kind of semicontinuity of the sets $((V,V))$, and in also of the stability subgroups.
The forward direction is \cite[Proposition~1.1.6]{palais}.

\begin{prop}
A $G$-space $X$ is locally proper if and only if for all $x\in X$, the isotropy subgroup $G_x$ is compact and for every neighborhood $U$ of $G_x$ there is a neighborhood $V$ of $x$ such that $((V,V))\subset U$.
\end{prop}

\begin{proof}
First assume that the action is locally proper.
We argue by contradiction.
Suppose we have $x\in X$ and a neighborhood $U$ of $G_x$ such that for every neighborhood $V$ of $x$ there exists $s\in ((V,V))$ such that $s\notin U$.
Fix a neighborhood $R$ of $x$ such that $((R,R))$ is compact.
Restricting to neighborhoods $V$ of $x$ with $V\subset R$,
we see that
we can find nets $\{s_i\}$ in the complement $U^c$ and $\{y_i\}$ in $R$ such that
\begin{itemize}
\item $s_iy_i\in R$ for all $i$,
\item $y_i\to x$, and
\item $s_iy_i\to x$.
\end{itemize}
Then $s_i\in ((R,R))$ for all $i$,
so passing to subnets and relabeling we can assume that $s_i\to s$ for some $s\in G$.
Then $s_iy_i\to sx$, so $sx=x$.
Thus $s\in G_x$.
But then eventually $s_i\in U$, which is a contradiction.

Conversely, assume the condition regarding isotropy groups and neighborhoods thereof,
and let $x\in X$.
Since $G_x$ is compact, we can choose a compact neighborhood $U$ of $G_x$,
and then we can choose a neighborhood $V$ of $x$ such that $((V,V))\subset U$.
Then $((V,V))$ is relatively compact,
and we have shown that the action is locally proper.
\end{proof}

The following result is contained in \cite[Theorem~1.2.9]{palais}.

\begin{prop}
A $G$-space $X$ is proper if and only if
it is locally proper and
$G\under X$ is Hausdorff.
\end{prop}

\begin{proof}
First assume that the action is proper.
Then it is locally proper, and
to show that $G\under X$ is Hausdorff, we will prove that if a net $\{Gx_i\}$ in $G\under X$ converges to both $Gx$ and $Gy$ then $Gx=Gy$.
Since the quotient map $X\to G\under X$ is open,
we can pass to a subnet and relabel so that without loss of generality $x_i\to x$.
Then again passing to a subnet and relabeling we can find $s_i\in G$ such that $s_ix_i\to y$.
Choose 
compact
neighborhoods $U$ of $x$ and $V$ of $y$,
so
that $((U,V))$ is compact by properness.
Without loss of generality $x_i\in U$ and $s_ix_i\in V$ for all $i$.
Then $s_i\in ((U,V))$ for all $i$,
so by compactness we can pass to subnets and relabel so that $\{s_i\}$ converges to some $s\in G$.
Then $s_ix_i\to sx$, so $sx=y$, and hence $Gx=Gy$.

Conversely, assume that the action is locally proper and $G\under X$ is Hausdorff.
Let $x,y\in X$.
By assumption we can choose a compact neighborhood $U$ of $x$ such that $((U,U))$ is compact.
Now choose any compact neighborhood $V$ of $y$.
To show that the action is proper, we will prove that $((U,V))$ is compact.
Let $\{s_i\}$ be any net in $((U,V))$.
For each $i$ choose $x_i\in U$ such that $s_ix_i\in V$.
By compactness we can pass to subnets and relabel so that $x_i\to z$ and $s_ix_i\to w$ for some $z\in U$ and $w\in V$.
Then by Hausdorffness we can write
\[
Gz=\lim Gx_i=\lim Gs_ix_i=Gw,
\]
so we can choose $s\in G$ such that $w=sz$.
Then $s_ix_i\to sz$, so
\[
s\inv s_ix_i\to z.
\]
Without loss of generality, for all $i$ we can assume that $s\inv s_ix_i\in U$,
so that $s\inv s_i\in ((U,U))$.
By compactness we can pass to subnets and relabel so that $s\inv s_i\to t$ for some $t\in G$.
Thus $s_i\to st$, and we have found a convergent subnet of $\{s_i\}$.
Thus $((U,V))$ is compact.
\end{proof}

\begin{ex}\label{locally proper not proper}
It is a well-known fact in topological dynamics that
there are actions
that are 
locally proper but not proper,
e.g., the action of $\Z$ on 
\[
[0,\infty)\times [0,\infty)\minus \{(0,0)\}
\]
generated by the homeomorphism $(x,y)\mapsto (2x,y/2)$, where any compact neighborhood of $\{(1,0),(0,1)\}$ meets itself infinitely often.
This action is locally proper because its restriction to each of the open sets
$[0,\infty)\times (0,\infty)$ and 
$(0,\infty)\times [0,\infty)$,
which cover the space,
are proper.
A closely related example is given by letting $\R$ act on the same space by
$s(x,y)=(e^sx,e^{-s}y)$.
\end{ex}

\begin{defn}
A $G$-space $X$ is \emph{pointwise proper} if for all $x\in X$ and compact $K\subset X$,
the set $((x,K))$ is compact.
\end{defn}

The above properness condition does not seem to be very often studied in the dynamics literature, and the
term we use is not standard, as far as we have been able to determine.

It is obvious that the above definition can be reformulated as follows:

\begin{lem}\label{orbit proper}
A $G$-space $X$ is pointwise proper if and only if 
for every $x\in X$ the map $s\mapsto sx$ from $G$ to $X$
is proper.
\end{lem}

\begin{prop}\label{T1}
If a $G$-space $X$ is pointwise proper then orbits are closed, and hence $G\under X$ is $T_1$.
\end{prop}

\begin{proof}
Let $x\in X$, and suppose we have a net $\{s_ix\}$ in the orbit $Gx$ converging to $y\in X$.
Choose a compact neighborhood $U$ of $y$.
Without loss of generality, for all $i$ we have $s_ix\in U$,
and hence $s_i\in ((x,U))$.
This set is compact by pointwise properness, so passing to a subnet and relabeling we can assume that $s_i\to s$ for some $s\in G$.
Then $s_ix\to sx$, so $y=sx\in Gx$.
\end{proof}

\begin{notn}
For $x\in X$ let $G_x$ denote the isotropy subgroup.
\end{notn}

\begin{prop}
A $G$-space $X$ is pointwise proper if and only if for all $x\in X$
the isotropy subgroup $G_x$ is compact and
the map
$s\mapsto sx$ from $G$ to $Gx$
is relatively open,
equivalently,
the action of $G$ on the orbit $Gx$ is conjugate to the action on the homogeneous space $G/G_x$.
\end{prop}

\begin{proof}
First assume that the action is pointwise proper,
and let $x\in X$.
Then $G_x$ is trivially compact.
By homogeneity it suffices to show that the map
$s\mapsto sx$ from $G$ to $Gx$
is relatively open at $e$.
Let $W$ be a neighborhood of $e$.
Suppose that $Wx$ is not a relative neighborhood of $x$ in the orbit $Gx$.
Then we can choose a net $\{s_i\}$ in $G$ such that $s_ix\notin Wx$ and $s_ix\to x$.
Choose a neighborhood $U$ of $x$ such that $((U,U))$ is compact.
Without loss of generality, for all $i$ we have $s_ix\in U$, and so $s_i\in ((x,U))$.
By compactness we can pass to a subnet and relabel so that $s_i\to s$ for some $s\in G$.
Then $s_ix\to sx$. Thus $sx=x$, and so $s\in G_x$.
But then eventually $s_i\in WG_x$, which is a contradiction because
$WG_xx=Wx$.

The converse is obvious, since if $G_x$ is compact the action of $G$ on $G/G_x$ is proper.
\end{proof}

We will show that pointwise properness is weaker than local properness, but for this we need a version of \propref{T1} for local properness.
The following result is contained in \cite[Proposition~1.1.4]{palais}.

\begin{lem}
If a $G$-space $X$ is locally proper then orbits are closed.
\end{lem}

\begin{proof}
Let $x\in X$, and suppose we have a net $\{s_i\}$ in $G$ such that $s_ix\to y$.
Choose an open $G$-invariant subset $U$ containing $y$ on which $G$ acts properly.
Then the action of $G$ on $U$ is pointwise proper, so the orbit $Gx$ is closed in $U$,
and hence $y\in Gx$.
\end{proof}

\begin{cor}
If a $G$-space $X$ is locally proper then it is pointwise proper.
\end{cor}

\begin{proof}
Let $x\in X$.
Choose an open $G$-invariant set $U\subset X$ such that the action of $G$ on $U$ is proper.
Let $K\subset X$ be compact,
and put $L=K\cap Gx$.
Then $L$ is compact because $Gx$ is closed, and $L\subset U$.
Thus
$((x,K))=((x,L))$
is compact because $\{x\}$ and $L$ are compact subsets of $U$ and $G$ acts properly on $U$.
\end{proof}

\begin{ex}\label{pointwise not local}
This example is
taken from
\cite[Example~5 in Section~2]{varadarajan}.
Recall that in \exref{locally proper not proper} 
we had an action of $\Z$ on the space
\[
X=\bigl([0,\infty)\times [0,\infty)\bigr)\minus\{(0,0\}
\]
generated by the homeomorphism $(x,y)\mapsto (2x,y/2)$.
We form the quotient of $X$ by identifying
$\{0\}\times (0,\infty)$ with $(0,\infty)\times \{0\}$
via
\[
(0,y)\sim (1/y,0).
\]
Then the action descends to the identification space,
and the quotient action is pointwise proper but not locally proper.
\end{ex}

With suitable countability assumptions, there is a surprise:

\begin{cor}[Glimm]
Let $G$ act on $X$, and assume that $G$ and $X$ are second countable,
and that every isotropy subgroup is compact.
Then the following are equivalent:
\begin{enumerate}
\item the action is pointwise proper;
\item for all $x\in X$ the map $sG_x\mapsto sx$ from $G/G_x$ to $Gx$ is a homeomorphism;
\item $G\under X$ is $T_0$;
\item $G\under X$ is $T_1$;
\item every orbit is locally compact in the relative topology from $X$;
\item every orbit is closed in $X$.
\end{enumerate}
\end{cor}

\begin{proof}
Because we assume that the isotropy groups are compact, we know (1) $\iff$ (2).
Glimm \cite[Theorem~1]{glimm} proves that, in the second countable case,
(2) $\iff$ (3) $\iff$ (5).
We also know (1) $\then$ (6) $\then$ (4).
Finally, (4) $\then$ (3) trivially.
\end{proof}

\section{$C^*$-ramifications}\label{c star ramifications}

Let $X$ be a $G$-space, and let $\alpha$ be the associated action of $G$ on $C_0(X)$.
In this section we examine the ramifications for the action $\alpha$ of the various properness conditions covered in \secref{spaces}.
For the state of the art in the case of proper actions, see \cite{echemeproper}.

\begin{notn}
If $\psi:X\to Y$ is a continuous map between locally compact Hausdorff spaces,
define $\psi^*:C_0(Y)\to C_b(X)$ by $\psi^*(f)=f\circ\psi$.
\end{notn}

It is an easy exercise to show:

\begin{lem}\label{Cc}
For a continuous map $\psi:X\to Y$ between locally compact Hausdorff spaces,
the following are equivalent:
\begin{enumerate}
\item $\psi$ is proper
\item $\psi^*$ maps $C_0(Y)$ into $C_0(X)$
\item $\psi^*$ maps $C_c(Y)$ into $C_c(X)$.
\end{enumerate}
\end{lem}

\begin{prop}\label{proper C0(X)}
The $G$-space $X$ is proper if and only if for all $f,g\in C_0(X)$
the map $s\mapsto \alpha_s(f)g$ from $G$ to $C_0(X)$ vanishes at infinity.
\end{prop}

\begin{proof}
First assume that the action is proper.
Since $C_c(X)$ is dense in $C_0(X)$,
by continuity it suffices to show that for all $f,g\in C_c(X)$
the map continuous $s\mapsto \alpha_s(f)g$ from $G\to C_0(X)$ has compact support.
Define $f\times g\in C_c(X\times X)$ by
\[
f\times g(x,y)=f(x)g(y).
\]
Since the map $\phi:G\times X\to X\times X$ given by $g(s,x)=(sx,x)$ is proper,
we have $\phi^*(f\times g)\in C_c(G\times X)$,
so there exist compact sets $K\subset G$ and $L\subset X$ such that
for all $(s,x)\notin K\times L$ we have
\[
0=\phi^*(f\times g)(s,x)=f\times g(sx,x)=f(sx)g(x)=\bigl(\alpha_{s\inv}(f)g\bigr)(x).
\]
Since $s\notin K$ implies $(s,x)\notin K\times L$,
we see that
the map $s\mapsto \alpha_s(f)g$ is supported in the compact set $K\inv$.

Conversely, assume the condition regarding $\alpha_s(f)g$.
To show that the action is proper, we will show that the map $\phi$ is proper,
and by \lemref{Cc} it suffices to show that
if $h\in C_c(X\times X)$ then $\phi^*(h)\in C_c(G\times X)$.
The support of $h$ is contained in a product $M\times N$ for some compact sets $M,N\subset X$,
and we can choose $f,g\in C_c(X)$ with $f=1$ on $M$ and $g=1$ on $N$.
Then $h(f\times g)=h$,
so it suffices to show that $\phi^*(f\times g)$ has compact support.
By assumption the support $K$ of $s\mapsto \alpha_s(f)g$ is compact,
and letting $L$ be the support of $g$ we see that
for all $(s,x)$ not in the compact set $K\inv\times L$ we have
\[
\phi^*(f\times g)(s,x)=\bigl(\alpha_{s\inv}(f)g\bigr)(x)=0.
\qedhere
\]
\end{proof}

\begin{prop}\label{pointwise proper C0(X)}
The $G$-space $X$ is pointwise proper if and only if for all $f\in C_0(X)$ and $\mu\in M(X)=C_0(X)^*$ the map
\[
g(s)=\int_Xf(sx)\,d\mu(x)
\]
is in $C_0(G)$.
\end{prop}

\begin{proof}
First assume that the action is pointwise proper.
Let $f\in C_0(X)$ and $\mu\in M(X)$, and define $g$ as above.
Note that $g$ is continuous since the 
associated action $(C_0(X),\alpha)$ is strongly continuous.
Suppose that $g$ does not vanish at $\infty$,
and pick $\epsilon>0$ such that the closed set
\[
S:=\{s\in G:|g(s)|\ge \epsilon\}
\]
is not compact.
It is a routine exercise to verify that we can find
a sequence $\{s_n\}$ in $S$
and
a compact neighborhood $V$ of $e$ such that
the sets $\{s_nV\}$ are pairwise disjoint.
Then for each $x\in X$ we have $\lim_{n\to\infty}f(s_nx)=0$,
because for fixed $x$ and any $\delta>0$ 
it is an easy exercise to see that
the compact set $\{s\in G:|f(sx)|\ge \delta\}$ can only intersect finitely many of the sets $\{s_nV\}$.
Thus by the Dominated Convergence theorem $\lim_{n\to\infty}g(x_n)=0$,
contradicting $s_n\in S$ for all $n$.

The converse follows immediately by taking $\mu$ to be a Dirac measure and applying \lemref{orbit proper}.
\end{proof}

\propref{module} below is the first time we need vector-valued integration.
There are numerous references dealing with this topic.
We are interested in integrating functions $f:\Omega\to B$, where $\Omega$ is a locally compact Hausdorff space equipped with a Radon measure $\mu$ (sometimes complex, but other times positive, and then frequently infinite), and $B$ is a Banach space.
Rieffel \cite[Section~1]{integrable} handles continuous bounded functions to a $C^*$-algebra using $C^*$-valued weights.
Exel \cite[Section~2]{exelunconditional} develops a theory of \emph{unconditionally integrable} functions with values in a Banach space, 
involving convergence of the integrals over relatively compact subsets of $G$.
Williams \cite[Appendix~B.1]{danacrossed} gives an exposition of the general theory of $L^1(\Omega,B)$, that 
in some sense
unifies the treatments in
\cite[Chapter~3]{dunford},
\cite{bourbaki_integration},
\cite[Chapter~II]{fd1},
and
\cite[part~I, Section~III.1]{hille}.
However, Williams uses a positive measure throughout, and we occasionally need complex measures;
this poses no problem, since the theory of \cite{danacrossed} can be applied to the positive and negative variations of the real and imaginary parts of a complex measure.
We prefer to use \cite{danacrossed} as our reference for vector-valued integration,
mainly because it entails \emph{absolute integrability} rather than \emph{unconditional integrability}
(see the first item in the following list).
Here are the main properties of $L^1(\Omega,B)$ that we need:
\begin{itemize}
\item The map $f\mapsto \int_\Omega f\,d\mu$ from $L^1(\Omega,B)$ to $B$ is bounded and linear,
where $\|f\|_1=\int_\Omega \|f(x)\|\,d|\mu|(x)$.

\item If $f\in L^1(\Omega,B)$ and $\omega$ is a bounded linear functional on $B$,
then $\omega\circ f\in L^1(\Omega)$ and
\[
\omega\left(\int_\Omega f(x)\,d\mu(x)\right)=\int_\Omega \omega(f(x))\,d\mu(x).
\]

\item If $f\in L^1(\Omega)$ and $b\in B$ then
\[
\int_\Omega (f\otimes b)\,d\mu=\left(\int_\Omega f\,d\mu\right)b,
\]
where $(f\otimes b)(x)=f(x)b$.

\item Every continuous bounded function from $\Omega$ to $B$ is measurable,
and is also essentially-separably valued on compact sets,
and so is integrable with respect to any complex measure.
\end{itemize}
Of course, we refer to the elements of $L^1(\Omega,B)$ as the \emph{integrable} functions from $\Omega$ to $B$.

If $X$ is a $G$-space, then
$C_0(X)$ gets a Banach-module structure over $M(G)=C_0(G)^*$ by
\[
\mu*f(x)=\int_G f(sx)\,d\mu(s)\righttext{for}\mu\in M(G),f\in C_0(X),x\in X.
\]
Here we are integrating the continuous bounded function
$s\mapsto \alpha_s(f)$
with respect to the complex measure $\mu$.

The following is a special case of \propref{module noncom} below.

\begin{prop}\label{module}
The action on $X$ is pointwise proper
if and only if
for each $f$ the map $\mu\mapsto \mu*f$ is weak*-to-weakly continuous.
\end{prop}

\section{Properness conditions for actions on $C^*$-algebras}\label{c star general}

Propositions~\ref{proper C0(X)} and \ref{pointwise proper C0(X)}
motivate the following:

\begin{defn}\label{s-proper}
An action $(A,\alpha)$ is \emph{s-proper} if for all $a,b\in A$ the map $s\mapsto \alpha_s(a)b$ from $G$ to $A$ vanishes at infinity.
\end{defn}

Taking adjoints, we see that the above map could equally well be replaced by $s\mapsto a\alpha_s(b)$.

\begin{defn}\label{w-proper}
An action $(A,\alpha)$ is \emph{w-proper} if for all $a\in A$ and all $\omega\in A^*$ the map
\[
g(s)=\omega\bigl(\alpha_s(a)\bigr)
\]
is in $C_0(G)$.
\end{defn}

We use the admittedly nondescriptive terminology s-proper and w-proper to avoid confusion with the myriad other uses of the word ``proper'' for actions on $C^*$-algebras.

\begin{rem}\label{locally proper C0(X)}
It is 
almost obvious
that 
a
$G$-space $X$
is locally proper if and only if
there is a family of $\alpha$-invariant closed ideals of $C_0(X)$
that densely span $C_0(X)$ and
on each of which $\alpha$ has the property in \propref{proper C0(X)}.
In fact, we will use this in the proof of \propref{locally regular}.
This could be generalized in various ways to actions on arbitrary $C^*$-algebras,
but since
we have no applications of this 
we will not 
pursue it here.
\end{rem}

Propositions~\ref{proper C0(X)} and \ref{pointwise proper C0(X)}
can be rephrased as follows:

\begin{cor}\label{proper action}
A $G$-space $X$ is proper if and only if the associated action $(C_0(X),\alpha)$ is s-proper,
and is pointwise proper if and only if $\alpha$ is w-proper.
\end{cor}

\begin{rem}\label{s vs w}
If an action $(A,\alpha)$ is s-proper then it is w-proper, since by the Cohen-Hewitt factorization theorem every functional in $A^*$ can be expressed in the form $\omega\cdot a$, where
\[
\omega\cdot a(b)=\omega(ab)\righttext{for}\omega\in A^*,a,b\in A.
\]
On the other hand, \exref{locally proper not proper} implies that $\alpha$ can be w-proper but not s-proper.
\end{rem}

If $(A,\alpha)$ is an action then $A$ gets a Banach module structure over $M(G)$ by
\[
\mu*a=\int_G \alpha_s(a)\,d\mu(s)\righttext{for}\mu\in M(G),a\in A.
\]

\propref{module} is the commutative version of the following:

\begin{prop}\label{module noncom}
An action $(A,\alpha)$ is w-proper if and only if for each $a\in A$
the map $\mu\mapsto \mu*a$ is weak*-to-weakly continuous.
\end{prop}

\begin{proof}
First assume that $\alpha$ is w-proper,
and let $a\in A$.
Let $\mu_i\to 0$ weak* in $M(G)$,
and let $\omega\in A^*$.
Then
\[
\omega(\mu_i*a)=\omega\left(\int_G \alpha_s(a)\,d\mu_i(s)\right)=\int \omega(\alpha_s(a))\,d\mu_i(s)
\to 0,
\]
because the map $s\mapsto \omega(\alpha_s(a))$ is in $C_0(G)$.

Conversely, assume the weak*-weak continuity,
and let $a\in A$ and $\omega\in A^*$.
If $\mu_i\to 0$ weak* in $M(G)$, then
\[
\int_G \omega(\alpha_s(a))\,d\mu_i(s)
=\omega(\mu_i*a)
\to 0
\]
by continuity.
By the well-known \lemref{in C0} below, the element $s\mapsto \omega(\alpha_s(a))$ of $C_b(G)$ lies in $C_0(G)$.
\end{proof}

In the above proof we appealed to the following well-known fact:

\begin{lem}\label{in C0}
Let $f\in C_b(G)$.
Then $f\in C_0(G)$ if and only if for every net $\{\mu_i\}$ in $M(G)$ converging weak* to 0 we have
\[
\int f\,d\mu_i\to 0.
\]
\end{lem}

s-properness and w-properness are both preserved by morphisms:

\begin{prop}\label{morphism}
Let $\phi:A\to M(B)$ be a nondegenerate homomorphism that is equivariant for actions $\alpha$ and $\beta$, respectively.
If $\alpha$ is 
s-proper
or w-proper,
then $\beta$ has the same property.
\end{prop}

\begin{proof}
First assume that $\alpha$ is s-proper.
Let $c,d\in B$.
By the Cohen-Hewitt Factorization theorem, $c=c'\phi(a)$ and $d=\phi(b)d'$ for some $a,b\in A$ and $c',d'\in B$.
Then
\begin{align*}
\beta_s(c)d
&=\beta_s(c'\phi(a))\phi(b)d'
\\&=\beta_s(c')\phi\bigl(\alpha_s(a)b\bigr)d',
\end{align*}
which vanishes at infinity because 
$s\mapsto \alpha_s(a)b$ does and $s\mapsto \beta_s(c')$ is bounded.

Now assume that $\alpha$ is w-proper.
Let $b\in B$ and $\omega\in B^*$.
We must show that the function $s\mapsto \omega\circ \beta_s(b)$ vanishes at $\infty$, and it suffices to do this for $\omega$ positive.
By the Cohen-Hewitt Factorization theorem we can assume that $b=\phi(a^*)c$ with $a\in A$ and $c\in B$.
By the Cauchy-Schwarz inequality for positive functionals on $C^*$-algebras, we have
\begin{align*}
\bigl|\omega\circ \beta_s(b)\bigr|^2
&=\Bigl|\omega\bigl(\phi(\alpha_s(a^*))\beta_s(c)\bigr)\Bigr|^2
\\&\le \omega\circ\phi(\alpha_s(a^*a))\omega(\beta_s(c^*c)),
\end{align*}
which vanishes at $\infty$ since $s\mapsto \omega\circ\phi(\alpha_s(a^*a))$ does and $s\mapsto \omega(\beta_s(c^*c))$ is bounded.
\end{proof}

In \secref{coactions} we will discuss properness for coactions, the dualization of actions.
Here we record an easy corollary of \propref{morphism} that involves coactions,
because it gives a rich supply of s-proper actions.
For now we just need to recall that if $(A,\delta)$ is a coaction of $G$,
with crossed product $C^*$-algebra $A\rtimes_\delta G$,
then there is a pair of nondegenerate homomorphisms
\[
\xymatrix{
A \ar[r]^-{j_A}
&M(A\rtimes_\delta G)
&C_0(G) \ar[l]_-{j_G}
}
\]
such that $(j_A,j_G)$ is a universal covariant homomorphism.
The \emph{dual action} $\what\delta$ of $G$ on $A\rtimes_\delta G$ is characterized by
\begin{align*}
\what\delta_s\circ j_A&=j_A
\\
\what\delta_s\circ j_G&=j_G\circ\rt_s,
\end{align*}
where $\rt$ is the action of $G$ on $C_0(G)$ by right translation.

\begin{cor}\label{dual action}
Every dual action is s-proper.
\end{cor}

\begin{proof}
If $\delta$ is a coaction of $G$ on $A$, then the canonical nondegenerate homomorphism $j_G:C_0(G)\to M(A\rtimes_\delta G)$ is $\rt-\what\delta$ equivariant.
Thus $\what\delta$ is s-proper since $\rt$ is.
\end{proof}

\cite[Corollary~5.9]{brogue} says that
if an action of a discrete group $G$ on a compact Hausdorff space $X$
is a-T-menable in the sense of \cite[Definition~5.5]{brogue},
then every covariant representation of the associated action $(C(X),\alpha)$ is weakly contained in a representation $(\pi,U)$, on a Hilbert space $H$, such that for all $\xi,\eta$ in a dense subspace of $H$ the function $s\mapsto (U_s\xi,\eta)$ is in $c_0(G)$.
The following proposition shows that w-proper actions on arbitrary $C^*$-algebras have a quite similar 
property:

\begin{prop}\label{c0 coef}
Let $(A,\alpha)$ be a w-proper action,
let $\pi$ be a representation of $A$ on a Hilbert space $H$,
and for each $s\in G$ suppose we have a unitary operator $U_s$ on $H$ such that $\ad U_s\circ\pi=\pi\circ\alpha_s$.
Then 
for all $\xi,\eta\in H$ the function
\[
s\mapsto \<U_s\xi,\eta\>
\]
vanishes at infinity.
\end{prop}

\begin{proof}
We can assume that $\pi$ is nondegenerate.
Then we can factor $\xi=\pi(a)\xi'$ for some $a\in A,\xi'\in H$, and
we have
\begin{align*}
|\<U_s\pi(a)\xi',\eta\>|
&=|\<U_s\xi',\pi(\alpha_s(a^*))\eta\>|
\\&\le \|\xi'\|\<\pi(\alpha_s(aa^*)\eta,\eta\>^{1/2},
\end{align*}
so we can appeal to w-properness with $\omega\in A^*$ defined by
\[
\omega(b)=\<\pi(b)\eta,\eta\>.
\qedhere
\]
\end{proof}

\begin{rem}
Note that in the above proposition we do not require $U$ to be a homomorphism; it could be a projective representation.
\end{rem}

\begin{rem}
Thus it would be interesting to study the relation between a-T-menable actions in the sense of \cite{brogue} and pointwise proper actions.
As it stands, the connection would be subtle, because an infinite discrete group cannot act pointwise properly on a compact space.
\end{rem}

\subsection*{Action on the compacts}

The following gives a strengthening of a special case of 
\propref{c0 coef}:
\begin{prop}\label{c0 coef 2}
Let $H$ be a Hilbert space, and let $\alpha$ be an action of $G$ on $\KK(H)$.
For each $s\in G$ choose a unitary operator $U_s$ such that $\alpha_s=\ad U_s$.
The following are equivalent:
\begin{enumerate}
\item $\alpha$ is s-proper;\label{one}
\item $\alpha$ is w-proper;\label{two}
\item $s\mapsto \<U_s\xi,\xi\>$ 
vanishes at infinity
for all $\xi\in H$.\label{three}
\item $s\mapsto \<U_s\xi,\eta\>$ 
vanishes at infinity
for all $\xi,\eta\in H$.\label{four}
\end{enumerate}
\end{prop}

\begin{proof}
We know \eqref{one} $\then$ \eqref{two} $\then$ \eqref{three} by
\remref{s vs w} and \propref{c0 coef},
and \eqref{three} $\then$ \eqref{four} by polarization.

Assume \eqref{four}.
Let $E(\xi,\eta)$ be the rank-1 operator given by $\zeta\mapsto \<\zeta,\eta\>\xi$.
For $\xi,\eta,
\gamma,\kappa
\in H$,
A routine computation shows
\begin{align*}
E(\xi,\eta)\alpha_s(E(\gamma,\kappa))
&=\<U_s\gamma,\eta\>E(\xi,\kappa)U_s^*,
\end{align*}
so
\[
\bigl\|E(\xi,\eta)\alpha_s(E(\gamma,\kappa))\bigr\|
\le \bigl|\<U_s\gamma,\eta\>\bigr|\|E(\xi,\kappa)\|,
\]
which vanishes at infinity.
Thus $s\mapsto a\alpha_s(b)$ is in $C_0(G,\KK(H))$ whenever $a$ and $b$ are rank-1,
and by linearity and density it follows that $\alpha$ ia s-proper.
\end{proof}

In \propref{c0 coef 2}, when $U$ can be chosen to be a representation of $G$,
we have the following:

\begin{cor}\label{cyclic}
Let $U$ be a representation of $G$ on a Hilbert space $H$, and 
let $\alpha=\ad U$ be the associated action of $G$ on $\KK(H)$.
Suppose that $\xi$ is a cyclic vector for the representation $U$.
If $s\mapsto \<U_s\xi,\xi\>$ vanishes at infinity, then $\alpha$ is s-proper.
\end{cor}

\begin{proof}
As in \cite[Remark~2.7]{brogue}, it is easy to see that
for all $\eta,\kappa$ in the dense subspace
of $H$
spanned by $\{U_s\xi:s\in G\}$
the function $s\mapsto \<U_s\eta,\kappa\>$ vanishes at infinity.
Then for all $\eta,\kappa\in H$ we can find sequences $\{\eta_n\},\{\kappa_n\}$ such that
$\|\eta_n-\eta\|\to 0$,
$\|\kappa_n-\kappa\|\to 0$, and
for all $n$ the function $s\mapsto \<U_s\eta_n,\kappa_n\>$ vanishes at infinity.
Then a routine estimation shows that
the functions $s\mapsto \<U_s\eta_n,\kappa_n\>$
converge uniformly to the function $s\mapsto \<U_s\eta,\kappa\>$,
and hence this latter function vanishes at infinity.
The result now follows from \propref{c0 coef 2}.
\end{proof}

\section{Rieffel properness}\label{rieffel proper}

We will show that if an action $(A,\alpha)$ is proper in Rieffel's sense
\cite[Definition~1.2]{proper} (see also \cite[Definition~4.5]{integrable}
then it is s-proper.
Rieffel's definitions of proper action in both of the above papers involve integration of $A$-valued functions on $G$,
and we have recorded our conventions regarding vector-valued integration
in the discussion preceding \propref{module}.
In \cite{proper}, Rieffel defined an action $(A,\alpha)$ to be \emph{proper}
(and we follow \cite{BusEch4} in using the term \emph{Rieffel proper})
if
$s\mapsto \alpha_s(a)b$ is integrable
for all $a,b$ in some dense subalgebra,
plus other conditions that we will not need.

\begin{cor}\label{integrable then proper}
Let $(A,\alpha)$ be an action.
\begin{enumerate}
\item Suppose that there is a dense $\alpha$-invariant subset $A_0$ of $A$ such that for all $a,b\in A_0$ the function
\begin{equation}\label{function}
s\mapsto \alpha_s(a)b
\end{equation}
is
integrable.
Then $\alpha$ is s-proper in the sense of \defnref{s-proper}.

\item
Suppose that there is a dense $\alpha$-invariant subset $A_0$ of $A$ such that for all $a\in A_0$
and all $\omega\in A^*$
the function 
\[
s\mapsto \omega(\alpha_s(a))
\]
is integrable.
Then $\alpha$ is w-proper in the sense of \defnref{w-proper}.
\end{enumerate}
\end{cor}

\begin{proof}
(1)
Since the functions \eqref{function} are uniformly continuous in norm, it follows immediately from 
the elementary lemma
\lemref{barbalat}
below
that
$s\mapsto \alpha_s(a)b$ is in $C_0(G,A)$ for all $a,b\in A_0$,
and then 
(1)
follows by density.

(2)
This can be proved similarly to (1), except now the functions are scalar-valued.
\end{proof}

In the above proof we referred to the following:

\begin{lem}\label{barbalat}
Let
$B$ be a Banach space, and let
$f:G\to B$ be uniformly continuous and integrable.
Then $f$ vanishes at infinity.
\end{lem}

\begin{proof}
Since the composition of $f$ with the norm on $B$ is uniformly continuous,
and $\|f\|_1=\int_G \|f(s)\|\,ds<\infty$ by hypothesis,
so this follows immediately from the scalar-valued case
(for which, see \cite[Theorem~1]{carcano}),
and which itself 
is a routine adaptation of a classical result about scalar-valued functions on $\R$,
sometimes referred to as Barbalat's Lemma.
\end{proof}

In the commutative case, \corref{integrable then proper} (1) has a converse.
First, following \cite{BusEch4}, we will call an action $(A,\alpha)$
\emph{Rieffel proper} if 
it satisfies the conditions of \cite[Definition~1.2]{proper}.

\begin{prop}\label{abel}
If $A=C_0(X)$ is commutative, then an action $(A,\alpha)$ is s-proper if and only if it is Rieffel proper.
\end{prop}

\begin{proof}
First assume that $\alpha$ is s-proper.
Then by \thmref{proper action} the $G$-space $X$ is proper,
and then it follows from \cite[Theorem~4.7 and and its proof]{integrable} that $\alpha$ is Rieffel proper.

Conversely, if $\alpha$ is Rieffel proper, then in particular it satisfies the hypothesis of \corref{integrable then proper} (1), so $\alpha$ is s-proper.
\end{proof}

\begin{rem}
Thus,
if the $G$-space $X$ is proper, then by \cite[Theorem~1.5]{proper}
(for the case of free action,
see also \cite[Situation~2]{rieffelapplications}, which refers to \cite{gre:smooth})
there is an ideal of $C_0(X)\rtimes_r G$
(which is known to equal $C_0(X)\rtimes G$ in this case --- see \propref{locally regular} below)
that is Morita equivalent to $C_0(G\under X)$.
This uses the following:
for $f\in C_c(X)$ the integral
\[
\what f(Gx):=\int_G f(sx)\,ds
\]
defines $\what f\in C_c(G\under X)$.
If the action on $X$ is just pointwise proper,
the integral $\int_G f(sx)\,ds$ still makes sense for $f\in C_c(X)$.
It would be interesting to know what properties persist in this case.
\end{rem}

\begin{ex}
\propref{abel} is not true for arbitrary actions $(A,\alpha)$.
For example, let $G$ be the free group $\F_n$ with $n>1$, and let $l$ be the length function.
Haagerup proves in \cite{haagerup} that for any $a>0$ the function $s\mapsto e^{-al(s)}$
is positive definite.

For $k\in\N$ define $h_k(s)=e^{-l(s)/k}$,
and let $U_k$ be the associated cyclic representation on a Hilbert space $H_k$,
so that we have a cyclic vector $\xi_k$ for $U_k$ with
\[
\<U_k(s)\xi_k,\xi_k\>=h_k(s).
\]
For each $k$, since $h_k$ vanishes at infinity the associated inner action $\alpha_k=\ad U_k$ of $G$ on $\KK(H_k)$ is s-proper, by \corref{cyclic}.

We claim that not all these actions $\alpha_k$ can be Rieffel proper.
Rieffel shows in \cite[Theorem~7.9]{integrable}
that the action $\alpha$ is proper in the sense of \cite[Definition~4.5]{integrable}
if and only if the representation $U$ is square-integrable in the sense of \cite[Definition~7.8]{integrable}.
This latter definition is somewhat nonstandard,
in that it uses concepts from the theory of left Hilbert algebras.
Also, Rieffel's definition of proper action in \cite{integrable} is somewhat complicated in that it involves $C^*$-valued weights.
In this paper we prefer to deal with the more accessible definition of Rieffel-proper action in \cite[Definition~1.2]{proper},
which Rieffel shows implies the properness condition \cite[Definition~4.5]{integrable}.
Actually, we need not concern ourselves here with Rieffel's definition of square-integrable representations, rather all we need is his reassurance (see \cite[Corollary~7.12 and Theorem~7.14]{integrable} that a cyclic representation of $G$ is square-integrable in his sense if and only if it is contained in the regular representation of $G$ --- so his notion of square integrability is equivalent to the more usual one (as he assures us in his comment following \cite[Definition~7.8]{integrable}).

Suppose that for every $k\in\N$ the action $\alpha_k$ of $G$ on $\KK(H_k)$ is Rieffel proper.
Then, as noted above, $\alpha_k$ is also proper in the sense of \cite[Definition~4.5]{integrable}, and so the representation $U_k$ is contained in the regular presentation $\lambda$.
Now we argue exactly as in \cite[proof of Proposition~4.4]{brogue}:
since the functions $h_k$ converge to 1 pointwise on the discrete group $G$,
for all $s\in G$ we have
\[
\<U_k(s)\xi_k,\xi_k\>\to 1,
\]
and hence
\[
\|U_k(s)\xi_k-\xi_k\|\to 0.
\]
Thus the direct sum representation $\bigoplus_k U_k$ weakly contains the trivial representation.
But since each $U_k$ is contained in $\lambda$, the direct sum is weakly contained in $\lambda$.
This gives a contradiction, since $G=\F_n$ is nonamenable.
\end{ex}

\section{Full equals reduced}\label{full=reduced}

\begin{defn}
Let $(A,\alpha)$ be an action.
We say the \emph{full and reduced crossed products of $(A,\alpha)$ are equal}
if the regular representation
\[
\Lambda:A\rtimes_\alpha G\to A\rtimes_{\alpha,r} G
\]
is an isomorphism.
\end{defn}

It is an old theorem \cite{phillipsproper}
that if $X$ is a second countable proper $G$-space then the associated action $(C_0(X),\alpha)$ has full and reduced crossed products equal.
It is folklore that the second-countability hypothesis can be removed --- see the proof of \propref{locally regular} and \remref{proper case}.
We extend this to pointwise proper actions and weaken the countability hypothesis:

\begin{thm}\label{pointwise regular}
If $X$ is a first countable pointwise proper $G$-space,
then the 
full and reduced crossed products of the
associated action $(C_0(X),\alpha)$ are equal.
\end{thm}

We need some properties of the ``full = reduced" phenomenon for actions.
First, it is
frequently
inherited by invariant subalgebras:

\begin{lem}\label{faithful}
Let $(A,\alpha)$ and $(B,\beta)$ be actions,
and let $\phi:A\to M(B)$ be an injective $\alpha-\beta$ equivariant homomorphism.
Suppose that the crossed-product homomorphism
\[
\phi\rtimes G:A\rtimes_\alpha G\to M(B\rtimes_\beta G)
\]
is faithful.
If the full and reduced crossed products of $\beta$ are equal,
then the full and reduced crossed products of $\alpha$ are equal.
\end{lem}

\begin{proof}
We have a commutative diagram
\[
\xymatrix@C+30pt{
A\rtimes_\alpha G \ar[r]^-{\phi\rtimes G} \ar[d]_{\Lambda_\alpha}
&M(B\rtimes_\beta G) \ar[d]^{\Lambda_\beta}
\\
A\rtimes_{\alpha,r} G \ar[r]_-{\phi\rtimes_r G}
&M(B\rtimes_{\beta,r} G),
}
\]
and the composition $\Lambda_\beta\circ(\phi\rtimes G)$ is faithful,
and therefore $\Lambda_\alpha$ is faithful.
\end{proof}

Next, 
``full = reduced'' is preserved by extensions:

\begin{lem}\label{extension}
Let $(A,\alpha)$ be an action, and let $J$ be a closed invariant ideal of $A$.
If the actions of $G$ on $J$ and on $A/J$ both
have full and reduced crossed products equal,
then the full and reduced crossed products of $\alpha$ are equal.
\end{lem}

\begin{proof}
Let $\phi:J\hookrightarrow A$ be the inclusion map, and let $\psi:A\to A/J$ be the quotient map.
We have a commutative diagram
\[
\xymatrix{
J\rtimes G \ar[r]^-{\phi\rtimes G} \ar[d]_{\Lambda_J}
&A\rtimes G \ar[r]^-{\psi\rtimes G} \ar[d]_{\Lambda_A}
&A/J\rtimes G \ar[d]^{\Lambda_{A/J}}
\\
J\rtimes_r G \ar[r]_-{\phi\rtimes_r G}
&A\rtimes_r G \ar[r]_-{\psi\rtimes_r G}
&A/J\rtimes_r G.
}
\]
The argument is a routine diagram-chase.
The vertical maps are the regular representations, which are surjective,
and moreover $\Lambda_J$ and $\Lambda_{A/J}$ are injective by hypothesis.
Since $J$ is an ideal, the map $\phi\rtimes G$ is an isomorphism onto
the kernel of $\psi\rtimes G$
\cite[Proposition~12]{gre:local}.
Further, since $J$ is an invariant subalgebra, $\phi\rtimes_r G$ is injective.
Let $x$ be in the kernel of $\Lambda_A$.
Then
\[
0=(\psi\rtimes_r G)\circ {\Lambda_A}(x)=\Lambda_{A/J}\circ (\psi\rtimes G)(x),
\]
so $x$ is in the kernel of $\psi\rtimes G$.
Thus $x\in J\rtimes G$, and
\[
0=\Lambda_A\circ (\phi\rtimes G)(x)=(\phi\rtimes_r G)\circ \Lambda_J(x),
\]
so $x=0$.
\end{proof}

Next we show that ``full = reduced'' is preserved by direct sums:

\begin{lem}\label{direct}
Let $\{(A_i,\alpha_i)\}_{i\in I}$ be a family of actions,
and assume that the full and reduced crossed products are equal for every $\alpha_i$.
Then the direct sum action 
\[
\left(\bigoplus_{i\in I} A_i,\bigoplus_{i\in I} \alpha_i\right)
\]
also has full and reduced crossed products equal.
\end{lem}

\begin{proof}
By \lemref{extension}, the conclusion holds if $I$ has cardinality 2, and by induction it holds if $I$ is finite.
By \cite[Proposition~12]{gre:local},
we can regard $(\bigoplus_{i\in I}A_i)\rtimes G$
as the inductive limit of the ideals $(\bigoplus_{i\in F}A_i)\rtimes G$ for finite $F\subset I$.
Similarly (but not requiring the reference to \cite{gre:local}),
we can regard $(\bigoplus_{i\in I}A_i)\rtimes_r G$
as the inductive limit of the ideals $(\bigoplus_{i\in F}A_i)\rtimes_r G$.
For every finite $F\subset I$ we have a commutative diagram
\[
\xymatrix{
\left(\bigoplus_{i\in F}A_i\right)\rtimes G \ar@{^(->}[r] \ar[d]_{\Lambda_F}^\simeq
&\left(\bigoplus_{i\in I}A_i\right)\rtimes G \ar[d]^{\Lambda_I}
\\
\left(\bigoplus_{i\in F}A_i\right)\rtimes_r G \ar@{^(->}[r]
&\left(\bigoplus_{i\in I}A_i\right)\rtimes_r G,
}
\]
where the vertical arrows are the regular representations.
Thus $\Lambda_I$ must be an isomorphism,
by properties of inductive limits.
\end{proof}

\begin{cor}\label{separate}
Let $(A,\alpha)$ be an action,
let
$\{(A_i,\alpha_i)\}_{i\in I}$ be a family of actions for which the full and reduced crossed products are equal,
and for each $i$ let $\phi_i:A\to M(A_i)$ be an $\alpha-\alpha_i$ equivariant homomorphism.
Let
\[
\phi:A\to M\left(\bigoplus_{i\in I}A_i\right)
\]
be the associated equivariant homomorphism.
Suppose that $\bigcap_{i\in I} \ker \phi_i=\{0\}$,
and that the crossed-product homomorphism
\[
A\rtimes_\alpha G\to M\left(\biggl(\bigoplus_{i\in I}A_i\biggr)\rtimes_{\alpha_i} G\right)
\]
is faithful.
Then $\alpha$ also has full and reduced crossed products equal.
\end{cor}

\begin{proof}
This follows immediately from Lemmas~\ref{faithful} and \ref{direct}.
\end{proof}

We are almost ready for the proof of \thmref{pointwise regular},
but first we need to recall the notion of quasi-regularity,
and we only need this in the special case of closed orbits:

\begin{defn}[{special case of \cite[Page~221]{gre:local}}]
Let $G$ act on $X$, and assume that all orbits are closed.
Then the associated action of $G$ on $C_0(X)$ is \emph{quasi-regular} 
if for every irreducible covariant representation $(\pi,U)$ of $(C_0(X),G)$ there is an orbit $G\cdot x$ such that
\[
\ker \pi=\{f\in C_0(X):f|_{G\cdot x}=0\}.
\]
\end{defn}

In this case, $\pi$ factors through a faithful representation $\rho$ of $C_0(G\cdot x)$ such that
the covariant pair $(\rho,U)$ is 
an irreducible representation of the restricted action $(C_0(G\cdot x),\alpha)$.
By \cite[Corollary~19]{gre:local},
the action is quasi-regular
if the orbit space $G\under X$ is second countable or 
\emph{almost Hausdorff} in the sense that every closed subset contains a dense relative open Hausdorff subset.
Here we will prove a variant of this result:

\begin{prop}\label{quasi}
If a $G$-space $X$ is pointwise proper and 
first countable,
then the associated action of $G$ on $C_0(X)$ is quasi-regular.
\end{prop}

We first need a topological property of pointwise proper actions on first countable spaces:

\begin{lem}\label{nbd}
If a $G$-space $X$ is pointwise proper and 
first countable,
then each orbit is a countable decreasing intersection of open $G$-invariant sets.
\end{lem}

\begin{proof}
Since orbits are closed, the quotient space $G\under X$ is $T_1$.
Since the quotient map is continuous and open, $G\under X$ is first countable.
In particular, every point is a countable decreasing intersection of open sets,
and the result follows.
\end{proof}

\begin{rem}
In \lemref{nbd} the first countability assumption could be weakened to: every point in $X$ is a $G_\delta$.
\end{rem}

It seems to us that the proof of \propref{quasi} is clearer if we separate out a special case:

\begin{lem}\label{special}
If a $G$-space $X$ is pointwise proper and 
first countable,
and if there is an irreducible covariant representation $(\pi,U)$ of $(C_0(X),G)$
such that $\pi$ is faithful,
then $X$ consists of a single orbit.
\end{lem}

\begin{proof}
We can extend $\pi$ to a representation of the algebra of bounded Borel functions on $X$,
and we let $P$ be the associated spectral measure 
(see, e.g., \cite[Theorem~2.5.5]{GM} for a version of the relevant theorem in the nonsecond-countable case; Murphy states the theorem for compact Hausdorff spaces, but it applies equally well to locally compact spaces by passing to the one-point compactification).
Since $(\pi,U)$ is irreducible, for every $G$-invariant Borel set $E$ we have $P(E)$ = 0 or 1.
In particular each orbit has spectral measure 0 or 1, and there can be at most one orbit with measure 1.

Claim: every nonempty $G$-invariant open subset $O$ of $X$ has spectral measure 1.
It suffices to show that $P(O)\ne 0$.
Since $O\ne\varnothing$, we can choose a nonzero $f\in C_0(X)$ supported in $O$.
Then
\[
0\ne \pi(f)=\pi(f\Chi_O)=\pi(f)P(O),
\]
so $P(O)\ne 0$.

Let $x\in X$. We will show that $X=G\cdot x$.
By \lemref{nbd} we can choose a decreasing sequence $\{O_n\}$ of open $G$-invariant sets with $\bigcap_1^\infty O_n=G\cdot x$.
By the 
properties of spectral measures, we have
\[
P(G\cdot x)=\lim_nP(O_n)=1.
\]
Thus every orbit has spectral measure 1, so there can be only one orbit.
\end{proof}

\begin{proof}[Proof of \propref{quasi}]
Let $(\pi,U)$ be an irreducible covariant representation of $(C_0(X),G)$ on a Hilbert space $H$.
Then $\ker\pi$ is a $G$-invariant ideal of $C_0(X)$, so there is a closed $G$-invariant subset $Y$ of $X$ such that
\[
\ker\pi=\{f\in C_0(X):f|_Y=0\}.
\]
We will show that $Y$ consists of a single orbit.
The restriction map $f\mapsto f|_Y$ is a $G$-equivariant homomorphism of $C_0(X)$ to $C_0(Y)$,
and $\ker\pi=C_0(X\minus Y)$,
so $\pi$ factors through a faithful representation $\rho$ of $C_0(Y)$
such that $(\rho,U)$ is an irreducible covariant representation of $(C_0(Y),G)$.
Then $Y$ is a single orbit, by \lemref{special}.
\end{proof}

\begin{proof}[Proof of \thmref{pointwise regular}]
For each $x\in X$,
the orbit $G\cdot x$ is closed,
the isotropy subgroup $G_x$ is compact,
and the canonical bijection $G/G_x\to G\cdot x$ is an equivariant homeomorphism.
Thus $G_x$ is in particular amenable, so
it follows from the above and
\cite[Corollary~4.3]{qs:regularity}
(see also \cite[Theorem~3.15]{K})
the associated action of $G$ on $C_0(G\cdot x)$ has full and reduced crossed products equal.
The restriction map $\phi_x:C_0(X)\to C_0(G\cdot x)$
is equivariant,
and 
we get an equivariant injective homomorphism
\[
\phi:C_0(X)\to M\left(\bigoplus_{x\in X}C_0(G\cdot x)\right).
\]

By \propref{quasi} the action of $G$ on $C_0(X)$ is quasi-regular, so every irreducible covariant representation of $(C_0(X),G)$ factors through a representation of $(C_0(G\cdot x),G)$ for some orbit $G\cdot x$.
It follows that
the crossed-product homomorphism
\[
\phi\rtimes G:C_0(X)\rtimes G\to M\left(\biggl(\bigoplus_{x\in X}C_0(G\cdot x)\biggr)\rtimes G\right)
\]
is faithful.
Therefore the theorem follows from \corref{separate}.
\end{proof}

The above strategy 
can also be used to prove the following folklore result,
which is a mild extension of Phillips' full-equals-reduced theorem.
Actually, we could not find the following result explicitly recorded in the literature, but it seems to us that it must have been noticed before.

\begin{prop}\label{locally regular}
If a $G$-space $X$ is locally proper then
the associated action $(C_0(X),\alpha)$ has full and reduced crossed products equal.
\end{prop}

Note that there is no countability hypothesis on $X$.

We need the following, which
will play a role similar to that of \corref{separate} in the pointwise proper case:

\begin{cor}\label{ideals}
Let $(A,\alpha)$ be an action, and let $\{J_i\}_{i\in I}$ be a family of $G$-invariant ideals that densely span $A$. 
If for every $i$ the restriction of the action to $J_i$ has full and reduced crossed products equal, then the action on $A$ has the same property.
\end{cor}

\begin{proof}
For each $i$ let $\alpha_i=\alpha|_{J_i}$,
let
$\phi_i:A\to M(J_i)$ be the $\alpha-\alpha_i$ equivariant homomorphism induced by the $A$-bimodule structure on $J_i$,
and let $\phi:A\to M(\bigoplus_{i\in I}J_i)$ be the associated equivariant homomorphism,
Since $A=\clspn_{i\in I}J_i$, we have $\bigcap_{i\in I}\ker\phi_i=\{0\}$.
Thus, by \corref{separate} we only need to show that
\[
\phi\rtimes G:A\rtimes_\alpha G\to M\left(\biggl(\bigoplus_{i\in I}J_i\biggr)\rtimes_{\alpha_i} G\right)
\]
is faithful.
Suppose that $\ker (\phi\rtimes G)\ne \{0\}$.
The ideals $J_i\rtimes_{\alpha_i} G$ densely span $A\rtimes_\alpha G$, since
the $J_i$'s densely span $A$.
Thus we can find $i\in J$ such that
\[
\{0\}\ne \ker(\phi\rtimes G)\cap (J_i\rtimes_{\alpha_i} G)=\ker(\phi|_{J_i}\rtimes G).
\]
But $\phi|_{J_i}\rtimes G$ is faithful since $\phi|_{J_i}$ is faithful and $J_i$ is a $G$-invariant ideal, so we have a contradiction.
\end{proof}

\begin{proof}[Proof of \propref{locally regular}]
First, if the $G$-space $X$ is actually proper,
then $G\under X$ is Hausdorff, so by \cite[Corollary~19]{gre:local} the action of $G$ on $C_0(X)$ is quasi-regular, so the conclusion follows as in the proof of \propref{pointwise regular}.
In the general case, $X$ is a union of open $G$-invariant subsets $U_i$, on each of which $G$ acts properly.
Then $C_0(X)$ is densely spanned by the ideals $C_0(U_i)$,
so by properness the associated actions $\alpha_i$ have full and reduced crossed products equal,
and hence the conclusion follows from \lemref{ideals}.
\end{proof}

\begin{rem}\label{proper case}
In the above proof we appealed to \cite[Corollary~19]{gre:local},
whose proof involved dense points in irreducible closed sets.
In the spirit of the techniques of the current paper, we offer an alternative argument: assume that $X$ is a proper $G$-space. To see that the action is quasi-regular,
as in the proof of \propref{quasi} we can assume without loss of generality that there is an irreducible covariant representation $(\pi,U)$ of $(C_0(X),G)$ such that $\pi$ is faithful.
We must show that $X$ consists of a single $G$-orbit.
Suppose $G\cdot x$ and $G\cdot y$ are distinct orbits in $X$.
By properness, the quotient space $G\under X$ is Hausdorff, so we can find disjoint open 
neighborhoods of $G\cdot x$ and $G\cdot y$ in $G\under X$,
and hence nonempty disjoint open $G$-invariant sets $U$ and $V$ in $X$.
But, as in the proof of \lemref{special}, letting $P$ denote the spectral measure associated to the representation $\pi$ of $C_0(X)$, every nonempty $G$-invariant open subset $O$ of $X$ has $P(O)=1$.
Since 
we cannot have two disjoint open sets with spectral measure 1,
we have a contradiction.
\end{rem}

The above methods quickly lead to another property of the crossed product.
Recall that a $C^*$-algebra is called \emph{CCR}, or \emph{liminal} \cite[Definition~4.2.1]{dix}, if every irreducible representation is by compacts.
In the second countable case, the following result is contained in \cite[Proposition~7.31]{danacrossed}.

\begin{prop}\label{ccr}
Let $X$ be a $G$-space.
In either of the following two situations, the crossed product $C_0(X)\rtimes G$ is CCR:
\begin{enumerate}
\item the action of $G$ is locally proper;
\item the action is pointwise proper and $X$ is first countable.
\end{enumerate}
\end{prop}

\begin{proof}
(1)
If the $G$-space $X$ is actually proper, then this is well-known.
To illustrate how the above methods apply, we give the following argument.
We have seen above that the action is quasi-regular,
and hence for every irreducible covariant representation $(\pi,U)$ of $(C_0(X),G)$
factors through an irreducible representation of the restriction of the action
to $(C_0(G\cdot x),G)$ for some $x\in X$.
The $G$-spaces $G\cdot x$ and $G/G_x$ are isomorphic,
and $C_0(G/G_x)\rtimes G$ is Morita equivalent to $C^*(G_x)$
by
Rieffel's version of Mackey's Imprimitivity Theorem \cite[Section~7]{rie:induced}.
Since the isotropy subgroup $G_x$ is compact, $C^*(G_x)$ is CCR, and hence the image of the integrated form $\rho\times U$, which equals the image of $\pi\times U$, is the algebra of compact operators.

In the general case, $X$ is a union of open $G$-invariant proper $G$-spaces $U_i$,
so $C_0(X)\rtimes G$ is the closed span of the CCR ideals $C_0(U_i)\rtimes G$.
Since every $C^*$-algebra has a largest CCR ideal \cite[Proposition~4.2.6]{dix},
$C_0(X)\rtimes G$ must be CCR.

(2)
By \propref{quasi} the action is quasi-regular,
and it follows as in part (1) that $C_0(X)\rtimes G$ is CCR.
\end{proof}

\begin{rem}
As remarked in
\cite[Example~2.7 (3)]{delaroche},
it follows from
\cite[Corollary~2.1.17]{ADR}
that if an action of $G$ on $X$ is proper then the action is amenable
(a condition involving approximation by positive-definite functions).
By \cite[Theorem~5.3]{delaroche}, if a $G$-space $X$ is amenable then the associated action $\alpha$ on $C_0(X)$ has full and reduced crossed products equal.
This raises a question:
is every pointwise proper action amenable?
It seems that amenability of the $G$-space is closely related to equality of full and reduced crossed products: by \cite[Theorem~3.3]{matsumura}, for an action of a discrete exact group $G$ on a compact space $X$,
if $\alpha$ has full and reduced crossed products equal then the action is amenable.
Unfortunately, this is of no help for our question, because a noncompact group cannot act pointwise properly on a compact space.
\end{rem}

\section{Properness conditions for coactions}\label{coactions}

We will now dualize the properness properties of Definitions~\ref{s-proper} and \ref{w-proper}.

To motivate how this will go, we pause to recall some basic facts regarding $C^*$-tensor products, commutative $C^*$-algebras, and actions.

For locally compact Hausdorff spaces $X,Y$ we have
the standard identifications
\[
C_0(X\times Y)=C_0(X)\otimes C_0(Y)
\]
and
\[
C_b(X)=M(C_0(X)).
\]

For a $C^*$-algebra $A$ we have
\[
A\otimes C_0(G)=C_0(G,A)
\]
and
\[
M(A\otimes C_0(G))=C_b(G,M^\beta(A)),
\]
where
$M^\beta(A)$ denotes the multiplier algebra
$M(A)$ 
with 
the strict topology.

For an action $(A,\alpha)$ we have a homomorphism
\[
\wilde\alpha:A\to M(A\otimes C_0(G))
\]
given by
\[
\wilde\alpha(f)(s,x)=\wilde\alpha(f)(s)(x)=f(sx)=\alpha_{s\inv}(f)(x).
\]
In fact, the image of $\wilde\alpha$ lies in the $C^*$-subalgebra
$\wilde M(A\otimes C_0(G))$,
where for any $C^*$-algebras $A$ and $D$
\[
\wilde M(A\otimes D):=\{m\in M(A\otimes D):m(1\otimes D)\cup (1\otimes D)m\subset A\otimes D\}.
\]

Using the above facts,
\corref{proper action} can be restated as follows:

\begin{lem}
An action $(A,\alpha)$ is
s-proper if and only if
\[
\wilde\alpha(A)(A\otimes 1_{M(C_0(G))})\subset A\otimes C_0(G),
\]
and is
w-proper if and only if
for all $\omega\in A^*$,
\[
(\omega\otimes\id)\circ\wilde\alpha(A)\subset C_0(G).
\]
\end{lem}

Now consider a coaction $(A,\delta)$ of $G$.
The main difference from actions is that the commutative $C^*$-algebra $C_0(G)$ is replaced by $C^*(G)$.
Here we will use the standard conventions for tensor products and coactions (see, e.g., \cite[Appendix~A]{enchilada},
in particular, the coaction is a homomorphism
\[
\delta:A\to \wilde M(A\otimes C^*(G)).
\]

\begin{defn}\label{proper coaction}
A coaction $(A,\delta)$ is \emph{s-proper} if
\[
\delta(A)\bigl(A\otimes 1_{M(C^*(G))}\bigr)\subset A\otimes C^*(G),
\]
and is \emph{w-proper} if
for all $\omega\in A^*$ we have
\[
(\omega\otimes\id)\circ\delta(A)\subset C^*(G).
\]
\end{defn}

\begin{rem}
In \cite[Definition~5.1]{exotic} we introduced
the above properness conditions,
but in that paper we used the term
\emph{proper coaction} for the above s-proper coaction,
and
\emph{slice proper coaction} for the above w-proper coaction (because it involves the slice map $\omega\otimes\id$).
After we submitted \cite{exotic}, we learned that
Ellwood had defined properness more generally for coactions of Hopf $C^*$-algebras
\cite[Definition~2.4]{ellwood}.
Indeed, 
\propref{proper C0(X)}
is essentially \cite[Theorem~2.9(b)]{ellwood}.
\defnref{proper coaction} should also be compared with
Condition~(A1) in \cite[Section~4.1]{Goswami-Kuku}, which concerns discrete quantum groups and involves the algebraic tensor product.
\end{rem}

\begin{rem}
An action on $C_0(X)$ can be w-proper without being s-proper,
and a fortiori a coaction can be w-proper without being s-proper, even for $G$ abelian.
\end{rem}

\begin{rem}
(1)
Just as every action of a compact group is s-proper, every coaction of a discrete group is s-proper, because then we in fact have $\delta(A)\subset A\otimes C^*(G)$.

(2)
For any locally compact group $G$ the canonical coaction $\delta_G$ on $C^*(G)$
given by the comultiplication
is s-proper, because it is symmetric in the sense that
\[
\delta_G=\Sigma\circ\delta_G,
\]
where $\Sigma$ is the flip automorphism on $C^*(G)\otimes C^*(G)$.
\end{rem}

If $(A,\delta)$ is a coaction, then
$A$ gets a Banach module structure over
the Fourier-Stieltjes algebra $B(G)=C^*(G)^*$ by
\[
f\cdot a=(\id\otimes f)\circ\delta(a)\righttext{for}f\in B(G),a\in A.
\]
In \cite[Lemma~5.2]{exotic} we proved the following dual analogue of \lemref{module noncom}

\begin{lem}\label{module coaction}
A coaction
$(A,\delta)$ is w-proper if and only if for all $a\in A$ the map $f\mapsto f\cdot a$ is 
weak*-to-weakly
continuous.
\end{lem}

\begin{proof}
See \cite[Lemma~5.2]{exotic}.
\end{proof}

s-properness and w-properness are both preserved by morphisms.
For w-properness this is proved in \cite[Proposition~5.3]{exotic}, and here it is for s-properness:

\begin{prop}\label{morphismco}
Let $\phi:A\to M(B)$ be a nondegenerate homomorphism that is equivariant for coactions $\delta$ and $\epsilon$, respectively.
If $\delta$ is 
s-proper
then $\epsilon$ has the same property.
\end{prop}

\begin{proof}
We have
\begin{align*}
(B\otimes 1)\epsilon(B)
&=(B\phi(A)\otimes 1)(\phi\otimes\id)(\delta(A))\epsilon(B)
\\&=(B\otimes 1)(\phi(A)\otimes 1)(\phi\otimes\id)(\delta(A))\epsilon(B)
\\&=(B\otimes 1)(\phi\otimes\id)\bigl((A\otimes 1)\delta(A)\bigr)\epsilon(B)
\\&\subset (B\otimes 1)(\phi\otimes\id)(A\otimes C^*(G))\epsilon(B)
\\&=(B\otimes C^*(G))\epsilon(B)
\\&\subset B\otimes C^*(G)).
\qedhere
\end{align*}
\end{proof}

\begin{cor}\label{dual coaction}
Every dual coaction is s-proper.
\end{cor}

\begin{proof}
If $(A,\alpha)$ is an action, then the canonical nondegenerate homomorphism $i_G:C^*(G)\to M(A\rtimes_\alpha G)$ is $\delta_G-\what\alpha$ equivariant,
where $\delta_G$ is the canonical coaction on $C^*(G)$ given by the comultiplication.
Thus $\what\alpha$ is s-proper since $\delta_G$ is.
\end{proof}

Recall that if $(A,\delta)$ is a coaction then the
\emph{spectral subspaces} $\{A_s\}_{s\in G}$ are given by
\[
A_s=\{a\in M(A):\delta(a)=a\otimes s\},
\]
and the
\emph{fixed-point algebra} is $A^\delta=A_e$.

\begin{prop}
Suppose $A\cap A^\delta\ne \{0\}$.
Then the following are equivalent:
\begin{enumerate}
\item
$\delta$ is s-proper;
\item
$\delta$ is w-proper;
\item
$G$ is discrete.
\end{enumerate}
\end{prop}

\begin{proof}
We know (1) implies (2) and (3) implies (1).
Assume (2), and let $a_e\in A\cap A^\delta$ be nonzero.
Then
\begin{align*}
f\mapsto
&f\cdot a_e
=(\id\otimes f)\circ\delta(a_e)
=(\id\otimes f)(a_e\otimes 1)
=f(e)a_e
\end{align*}
is weak*-weak continuous from $B(G)$ to $A$,
so $f\mapsto f(e)$ is a weak* continuous linear functional on $B(G)$,
which implies
$e\in C^*(G)$,
and hence
$G$ is discrete.
\end{proof}

\begin{rem}
Of course, the above proposition applies if $A$ is unital.
Also note that when $G$ is nondiscrete a coaction $(A,\delta)$ can be s-proper and still have nonzero spectral subspaces $A_s$ (and hence nontrivial fixed-point algebra $A^\delta$, but these will be subspaces in $M(A)$ that intersect $A$ trivially.
\end{rem}

For the next lemma, recall that if $(A,\delta)$ is a coaction, then
a projection $p\in M(A)$ is called \emph{$\delta$-invariant} if
$p\in A^\delta$,
and in this case $\delta$ restricts to a coaction $\delta_p$ on the corner $pAp$:
\[
\delta_p(pap)=(p\otimes 1)\delta(a)(p\otimes 1)\in M(pAp\otimes C^*(G))\righttext{for}a\in A.
\]

\begin{lem}\label{corner}
Let $(A,\delta)$ be a coaction, and let $p$ be a $\delta$-invariant projection in $M(A)$.
If $(A,\delta)$ is s-proper, then so is the corner coaction $(pAp,\delta_p)$ defined above.
\end{lem}

\begin{proof}
This is a routine computation:
\begin{align*}
\delta_p(pAp)(pAp\otimes 1)
&\subset (p\otimes 1)\delta(A)(A\otimes 1)(p\otimes 1)
\\&\subset (p\otimes 1)(A\otimes C^*(G))(p\otimes 1)
\\&=pAp\otimes C^*(G).
\qedhere
\end{align*}
\end{proof}

For the 
definitions of normalization and maximalization,
we refer to
\cite[Appendix~A.7]{enchilada}
and 
\cite{ekq}.
Normalizations 
and maximalizations
always exist, and are unique up to equivariant isomorphism.

\begin{prop}
For any coaction $(A,\delta)$, the following are equivalent:
\begin{enumerate}
\item $(A,\delta)$ is s-proper;
\item The normalization $(A^n,\delta^n)$ is s-proper;
\item The maximalization $(A^m,\delta^m)$ is s-proper.
\end{enumerate}
\end{prop}

\begin{proof}
It follows from \propref{morphism}
that (1) implies (2) and (3) implies (1),
and a careful examination of 
the construction of the maximalization in \cite{ekq}
(particularly
Lemma~3.6
and
the proof of Theorem~3.3
in that paper)
shows that
(2) implies (3).
\end{proof}

\begin{rem}
In case the above proof seems overly fussy, note that it would not be enough to observe that
the double-dual coaction $\what{\what\delta}$ is automatically s-proper
and the maximalization $\delta^m$ is Morita equivalent to $\what{\what\delta}$,
because s-properness is \emph{not} preserved by Morita equivalence ---
otherwise every coaction of an amenable group would be s-proper!
\end{rem}

Recall from \cite[Proposition~3.1]{kmqw1} that if $\AA\to G$ is a Fell bundle then there is a coaction $\delta_\AA$ of $G$ on the (full) bundle algebra $C^*(\AA)$.

\begin{prop}\label{fell proper}
Let $\AA\to G$ be a Fell bundle. Then the coaction $(C^*(\AA),\delta_\AA)$ is s-proper.
\end{prop}

\begin{proof}
We must show that for all $a,b\in C^*(\AA)$ we have $\delta(a)(b\otimes 1)\in C^*(\AA)\otimes C^*(G)$,
and by density and nondegeneracy it suffices to take $a\in \Gamma_c(\AA)$ and $b$ of the form $f\cdot b$ for $f\in A(G)\cap C_c(G)$:
\begin{align*}
\delta(a)(f\cdot b\otimes 1)
&=\int_G \bigl(a(t)f\cdot b\otimes t\bigr)\,dt
\\&=\int_G \bigl(a(t)b\otimes tf\bigr)\,dt\righttext{(justified below)}
\\&\in C^*(\AA)\otimes C^*(G),
\end{align*}
because the integrand
\[
t\mapsto a(t)b\otimes tf
\]
is in $C_c(G,C^*(\AA)\otimes C^*(G))$.
In the above computation we used the equality
\[
a(t)f\cdot b\otimes t=a(t)b\otimes tf\righttext{for all}t\in G,
\]
which we justify as follows:
computing inside $M(C^*(\AA)\otimes C^*(G))$, we have
\begin{align*}
a(t)f\cdot b\otimes t
&=\bigl(a(t)\otimes t\bigr)\bigl(f\cdot b\otimes 1\bigr)
\\&=\bigl(a(t)\otimes t\bigr)\bigl(b\otimes f\bigr)\righttext{(justified below)}
\\&=a(t)b\otimes tf,
\end{align*}
where we must now justify the equality $f\cdot b\otimes 1=b\otimes f$:
both sides can be regarded as compactly supported strictly continuous functions from $G$ to $M(C^*(\AA)\otimes C^*(G))$, and for all $s\in G$ we have
\begin{align*}
(f\cdot b\otimes 1)(s)
&=(f\cdot b)(s)\otimes 1
\\&=f(s)b(s)\otimes 1
\\&=b(s)\otimes f(s)\righttext{(since $f(s)\in\C$)}
\\&=(b\otimes f)(s).
\qedhere
\end{align*}
\end{proof}

\begin{rem}
Let $\AA$ be a Fell bundle over $G$, and let
\[
\delta_\AA^r=(\id\otimes\lambda)\circ\delta_\AA:C^*(\AA)\to M(C^*(\AA)\otimes C^*_r(G))
\]
be the reduction of the coaction $\delta_\AA$.
\cite[Theorem~3.10]{BussCoaction} shows that $\delta_\AA^r$ is \emph{integrable} in the sense that the set of positive elements $a$ in $A$ for which $\delta_\AA^r(a)$ is in the domain of the operator-valued weight $\id\otimes\varphi$ is dense in $A^+$, where $\varphi$ is the Plancherel weight on $C^*_r(G)$.
\end{rem}

\corref{dual integrable} below is a dual analogue of 
\corref{integrable then proper} (1).
To explain the terminology, we recall a few things from Buss' thesis \cite{bussthesis}.
Buss worked with reduced coactions, but as he points out in \cite[Remark~2.6.1 (4)]{bussthesis},
the theory carries over to full coactions by considering the reductions of the coactions.
Throughout, $(A,\delta)$ is a coaction of $G$.

Let $\varphi$ be the Plancherel weight on $C^*(G)$,
let $\MM_\varphi^+=\{c\in C^*(G)^+:\varphi(c)<\infty\}$,
$\NN_\varphi=\{c\in C^*(G):c^*c\in \MM_\varphi\}$,
and $\MM_\varphi=\spn \MM_\varphi^+$,
so that $\MM_\varphi^+$ is a hereditary cone in $C^*(G)$,
and coincides with both $\MM_\varphi\cap C^*(G)^+$
and $\spn \NN_\varphi^*\NN_\varphi$,
and $\varphi$ extends uniquely to a linear functional on $\MM_\varphi$.

Let $\id\otimes\varphi$ denote the associated $M(A)$-valued weight on $A\otimes C^*(G)$,
with associated objects
$\MM_{\id\otimes\varphi}^+$,
$\NN_{\id\otimes\varphi}$,
and $\MM_{\id\otimes\varphi}$,
and characterized as follows:
for $x\in (A\otimes C^*(G))^+$ we have
$x\in \MM_{\id\otimes \varphi}^+$ if and only if
there exists $a\in M(A)^+$ such that
\[
\theta(a)=(\id\otimes\varphi)\bigl((\theta\otimes\id)(x)\bigr)
\righttext{for all}\theta\in A^{*+},
\]
in which case $(\id\otimes\varphi)(x)=a$.
We have $(\id\otimes\varphi)(a\otimes c)=\varphi(c)a$ for all $a\in A$ and $c\in \MM_\varphi$.

Let $\Lambda:\NN_\varphi\to L^2(G)$ be the canonical embedding associated to the GNS construction for $\varphi$, so that
$\Lambda(bc)=\lambda(b)\Lambda(c)$ for all $b\in C^*(G)$ and $c\in \NN_\varphi$.

Let $\id\otimes\Lambda:\NN_{\id\otimes\varphi}\to M(A\otimes L^2(G))=\LL(A,A\otimes L^2(G))$
be the map associated to the KSGNS construction for $\id\otimes\varphi$,
characterized by
\[
(\id\otimes\Lambda)(x)^*\bigl(a\otimes\Lambda(c)\bigr)=(\id\otimes\varphi)\bigl(x^*(a\otimes c)\bigr)
\]
for all $x\in \NN_{\id\otimes\varphi}$, $a\in A$, and $c\in \NN_\varphi$.
We have $(\id\otimes\Lambda)(a\otimes c)=a\otimes \Lambda(c)$
for all $a\in A$ and $c\in \NN_\varphi$,
and
\[
(\id\otimes\Lambda)(xy)=(\id\otimes \lambda)(x)(\id\otimes\Lambda)(y)
\]
for all $x\in M(A\otimes C^*(G))$ and $y\in \NN_{\id\otimes\Lambda}$.

The weight $\varphi$ extends canonically to $M(C^*(G))$,
and the associated objects are denoted by
$\bar\MM_\varphi^+$, $\bar\NN_\varphi$, and $\bar\MM_\varphi$.
Similarly for the canonical extension of $\id\otimes\varphi$ to $M(A\otimes C^*(G))$,
$\bar\MM_{\id\otimes\varphi}^+$, etc.

Let
\[
A\si=\{a\in A:\delta(aa^*)\in \bar\MM_{\id\otimes\varphi}^+\}.
\]
Then the coaction $\delta$ is 
\emph{square-integrable} if $A\si$ is dense in $A$.
For $a\in A\si$ define
\[
\<\<a|\in M(A\otimes L^2(G))=\LL(A,A\otimes L^2(G))
\]
by
\[
\<\<a|(b)=(\id\otimes\Lambda)\bigl(\delta(a)^*(b\otimes 1)\bigr),
\]
then define $|a\>\>=\<\<a|^*\in \LL(A\otimes L^2(G),A)$,
and for $a,b\in A\si$ define
\[
\<\<a|b\>\>=\<\<a|\circ |b\>\>\in \LL(A\otimes L^2(G)).
\]
Then $(A,\delta)$ is \emph{continuously square-integrable} if there is a dense subspace $\RR\subset A\si$ such that
\[
\<\<a|b\>\>\in A\rtimes_\delta G\subset \LL(A\otimes L^2(G))
\righttext{for all}a,b\in A\si.
\]

\begin{cor}\label{dual integrable}
Every continuously square-integrable coaction is s-proper.
\end{cor}

\begin{proof}
Let $(A,\delta)$ be a continuously square-integrable coaction.
\cite[Section~6.8 and Proposition~6.9.4]{bussthesis} 
gives a
Fell bundle $\AA$ over $G$
and 
a $\delta_\AA-\delta$ equivariant surjective homomorphism $\kappa:C^*(\AA)\to A$.
By \propref{morphism}, every quotient of an s-proper coaction is s-proper, so the corollary follows from \propref{fell proper}.
\end{proof}

\begin{rem}
Buss states on page 10 of \cite{bussthesis} that it is an open problem whether $\delta_\AA$ is maximal,
but in the second-countable case this is now known to be true \cite[Theorem~8.1]{kmqw1}.
\end{rem}

\section{$E$-crossed products}\label{E}

For an action $(B,\alpha)$, there are numerous crossed-product $C^*$-algebras.
The largest is the (full) crossed product $B\rtimes_\alpha G$ and the smallest is the reduced crossed product $B\rtimes_{\alpha,r} G$.
But there are frequently many ``exotic'' crossed products in between,
i.e., quotients $(B\rtimes_\alpha G)/J$ where $J$ is a nonzero ideal properly contained in the kernel of the regular representation $\Lambda$.
In \cite{graded}, inspired by work of Brown and Guentner \cite{brogue}, we introduced a tool that produces many (but not all) of these exotica.
Our strategy is to base everything on ``interesting'' $C^*$-algebras $C^*(G)/I$ between $C^*(G)$ and $C^*_r(G)$.
We call a closed ideal $I$ of $C^*(G)$ \emph{small} if
it is contained in the kernel of the regular representation $\lambda$
and is \emph{$\delta_G$-invariant},
i.e.,
the coaction $\delta_G$ descends to a coaction on
$C^*(G)/I$.
In \cite[Corollary~3.13]{graded} we proved that
$I$ is small
if and only if the annihilator $E=I\ann$ in $B(G)$ is an ideal,
which will then be \emph{large} in the sense that it is nonzero, weak* closed, and $G$-invariant,
where $B(G)$ is given the $G$-bimodule structure
\[
(s\cdot f\cdot t)(u)=f(tus)\righttext{for}f\in B(G),s,t,u\in G.
\]
Large ideals automatically contain the reduced Fourier-Stieltjes algebra $B_r(G)=C^*_r(G)^*$ \cite[Lemma~3.14]{graded},
and the map $E\mapsto \pann E$ gives a bijection between the large ideals of $B(G)$ and the small ideals $I$ of $C^*(G)$.
For a large ideal $E$ the quotient map
\[
q_E:C^*(G)\to  C^*_E(G):=C^*(G)/\pann E
\]
is equivariant for $\delta_G$ and a coaction $\delta^E_G$.

\begin{ex}
$E=\bar{B(G)\cap C_0(G)}$ is a large ideal, and if $G$ is discrete then $G$ has the Haagerup property if and only if $E=B(G)$ \cite[Corollary~3.5]{brogue}.
\end{ex}

\begin{ex}
For $1\le p\le \infty$, $E^p:=\bar{B(G)\cap L^p(G)}$ is a large ideal.
Of course $E^\infty=B(G)$.
For $p\le 2$ we have $E^p=B_r(G)$ \cite[Proposition~4.2]{graded} (and \cite[Proposition~2.11]{brogue} for discrete $G$).
If $G=\F_n$ for $n>1$,
it has been attributed to Okayasu \cite{okayasu} and (independently) to Higson and Ozawa (see \cite[Remark~4.5]{brogue}) that
for $2\le p<\infty$ the ideals $E^p$ are all distinct.
\end{ex}

Given an action $(B,\alpha)$,
we use large ideals to produce exotic crossed products by involving the dual coaction $\what\alpha$ on $B\rtimes_\alpha G$.
As in \cite{exotic}, the process is most cleanly expressed in terms of an abstract coaction $(A,\delta)$.
An ideal $J$ of $A$ is called \emph{$\delta$-invariant} if $\delta$ descends to a coaction on the quotient $A/J$.
We call an ideal $J$ \emph{small} if it is invariant and contained in the kernel of $j_A$, where $(j_A,j_G)$ is the canonical covariant homomorphism of $(A,C_0(G))$ into the multiplier algebra of the crossed product $A\rtimes_\delta G$.
For the coaction $(C^*(G),\delta_G)$, this is consistent with the above notion of small ideals of $C^*(G)$.

Recall that $A$ gets a $B(G)$-module structure by
\[
f\cdot a=(\id\otimes f)\circ\delta(a)\righttext{for}f\in B(G),a\in A.
\]
For any large ideal $E$ of $B(G)$,
\[
\JJ(E)=\{a\in A:f\cdot a=0\text{ for all }f\in E\}
\]
is a small ideal of $A$ \cite[Observation~3.10]{exotic}.
For a dual coaction $(B\rtimes_\alpha G,\what\alpha)$,
we call the quotient
\[
B\rtimes_{\alpha,E} G:=(B\rtimes_\alpha G)/\JJ(E)
\]
an \emph{$E$-crossed product}.

In the other direction, for any small ideal $J$ of $A$,
\[
\EE(J)=\{f\in B(G):(s\cdot f\cdot t)\cdot a=0\text{ for all }a\in J,s,t\in G\}
\]
is an ideal of $B(G)$,
which is $G$-invariant by construction,
and which will be weak*-closed if the coaction is w-proper.
The following is \cite[Lemma~6.4]{exotic}:

\begin{lem}\label{galois ideal}
For any w-proper coaction $(A,\delta)$,
the above maps
$\JJ$ and $\EE$ form a Galois correspondence
between the large ideals of $B(G)$ and
the 
small ideals of $A$.
\end{lem}

By \emph{Galois correspondence} we mean that
$\JJ$ and $\EE$ reverse inclusions,
$E\subset \EE(\JJ(E))$ for every large ideal $E$ of $B(G)$,
and
$J\subset \JJ(\EE(J))$ for every small ideal $J$ of $A$.

Since every dual coaction is s-proper, and hence w-proper, \lemref{galois ideal}
is applicable to $(B\rtimes_\alpha G,\what\alpha)$ for any action $(B,\alpha)$.
In \cite[Theorem~6.10]{exotic} we used this Galois correspondence to exhibit examples of small ideals $J$ that are not of the form $\JJ(E)$ for any large ideal $E$.
Buss and Echterhoff \cite[Example~5.3]{BusEch} have given examples that are better in the sense that the coaction $(A,\delta)$ is of the form $(B\rtimes_\alpha G,\what\alpha)$.
Consequently, there are exotic crossed products that are not $E$-crossed products for any large ideal $E$.

However, the real goal is not to look at exotic crossed products one at a time, but rather all at once:
In \cite{bgwexact}, Baum, Guentner, and Willett
define a \emph{crossed-product} as a functor
$(B,\alpha)\mapsto B\rtimes_{\alpha,\tau} G$,
from the category of actions to the category of $C^*$-algebras,
equipped with natural transformations
\[
\xymatrix{
B\rtimes_\alpha G \ar[r] \ar[d]
&B\rtimes_{\alpha,\tau} G \ar[dl]
\\
B\rtimes_{\alpha,r} G,
}
\]
where the vertical arrow is the regular representation,
such that the horizontal arrow is surjective.

For a large ideal $E$ of $B(G)$, the $E$-crossed product
\[
(B,\alpha)\mapsto B\rtimes_{\alpha,E} G
\]
gives a crossed-product functor in the sense of \cite{bgwexact}.

\cite{bgwexact} defines a crossed-product functor $\tau$ to be
\emph{exact}
if for every short exact sequence
\[
0\to
(B_1,\alpha_1)\to
(B_2,\alpha_2)\to
(B_3,\alpha_3)\to
0
\]
of actions the corresponding sequence of $C^*$-algebras
\[
0\to
B_1\rtimes_{\alpha_1,\tau} G\to
B_2\rtimes_{\alpha_2,\tau} G\to
B_3\rtimes_{\alpha_3,\tau} G\to
0
\]
is exact,
and
\emph{Morita compatible}
if for every action $(B,\alpha)$ the canonical \emph{untwisting isomorphism}
\[
(B\otimes\KK_G)\rtimes G\simeq (A\rtimes G)\otimes \KK_G,
\]
where $\KK_G$ denotes the compact operators on $\bigoplus_{n=1}^\infty L^2(G)$,
descends to an isomorphism
\[
(B\otimes\KK_G)\rtimes_\tau G\simeq (A\rtimes_\tau G)\otimes \KK_G
\]
of $\tau$-crossed products.
\cite[Theorem~3.8]{bgwexact}
(with an assist from Kirchberg)
shows that there is a unique minimal exact and Morita compatible crossed product,
and \cite{bgwexact} uses this to give a promising reformulation of the Baum-Connes conjecture.

If $E$ is any large ideal of $B(G)$,
the $E$-crossed product
\[
(B,\alpha)\mapsto B\rtimes_{\alpha,E} G
\]
is a crossed-product functor in the sense of \cite{bgwexact},
and it is automatically Morita compatible \cite[Lemma~A.5]{bgwexact}.

It is an open problem whether the minimal functor of \cite{bgwexact} is an $E$-crossed product for some large ideal $E$.
The counterexamples of \cite{BusEch} do not necessarily give a negative answer, because it is unknown whether they fit into a crossed-product functor.
The state of the art regarding $E$-crossed products is depressingly meager at this early stage ---
we do not even know any examples
other than $B(G)$ itself
of large ideals $E$ for which the $E$-crossed-product functor is exact for all $G$!
Of course, by definition the $B_r(G)$-crossed product is exact for an \emph{exact} group $G$
(where $B_r(G)=C^*_r(G)^*$ denotes the reduced Fourier-Stieltjes algebra).
But nonexact groups are quite mysterious.



\providecommand{\bysame}{\leavevmode\hbox to3em{\hrulefill}\thinspace}
\providecommand{\MR}{\relax\ifhmode\unskip\space\fi MR }
\providecommand{\MRhref}[2]{%
  \href{http://www.ams.org/mathscinet-getitem?mr=#1}{#2}
}
\providecommand{\href}[2]{#2}

\end{document}